\documentclass[11pt]{amsart}

\usepackage{amsmath, amsthm, amsfonts, amssymb}
\newtheorem{prop}{Proposition}
\newtheorem{theorem}{Theorem}
\newtheorem{coroll}{Corollary}
\newtheorem{lemma}{Lemma}

\newtheorem{remark}{Remark}

\newcommand{\BRW}{\mbox{BRW}}
\newcommand{\ig}{[\hspace{-1.5pt}[}
\newcommand{\id}{]\hspace{-1.5pt}]}
\def\floor#1{\lfloor #1 \rfloor}
\def\ceil#1{\lceil #1 \rceil}
\newcommand{\Z}{\mathbb{Z}}
\newcommand{\N}{\mathbb{N}}
\renewcommand{\S}{\mathcal{S}}
\newcommand{\C}{\mathcal{C}}

\newcommand{\un}{\mathbf{1}}
\renewcommand{\P}{\mathbb{P}}

\newcommand{\E}{\mathbb{E}}
\newcommand{\R}{\mathbb{R}}

\newcommand{\F}{\mathcal{F}}

\author{Jean B\'{e}rard}
\address[Jean B\'{e}rard]{\noindent Institut Camille Jordan, UMR CNRS 5208, 43, boulevard du 11 novembre
1918, Villeurbanne, F-69622, France; universit\'{e} de Lyon, Lyon, F-69003, France; 
universit\'{e} Lyon 1, Lyon, F-69003, France
\newline
e-mail:  \rm \texttt{jean.berard@univ-lyon1.fr}}

\thanks{We would like to thank A. Fribergh and J.-B. Gou\'er\'e for useful discussions.}

\date{}

\title[An example of Brunet-Derrida behavior for a particle system]{An example of Brunet-Derrida behavior for a branching-selection particle system on $\Z$}

\begin{document}

\begin{abstract}
We consider a branching-selection particle system on $\Z$ with $N \geq 1$ particles. During a branching step, each particle is replaced by
two new particles, whose positions are shifted from that of the original particle by independently performing two random walk steps according to the distribution $p \delta_{1} + (1-p) \delta_{0}$, from 
the location of the original particle. During the selection step that follows, only the $N$ rightmost particles are kept among the $2N$ particles 
obtained at the branching step, to form a new population of $N$ particles. 
After a large number of iterated branching-selection steps, the displacement of the whole population of $N$ particles is ballistic, 
with deterministic asymptotic speed $v_{N}(p)$. As $N$ goes to infinity, $v_{N}(p)$ converges to a finite limit $v_{\infty}(p)$. 
The main result is that, for every $0<p<1/2$, as $N$ goes to infinity, the order of magnitude of the difference
$v_{\infty}(p)- v_{N}(p)$ is $\log(N)^{-2}$. This is called Brunet-Derrida behavior in reference to the 1997 paper by E. Brunet and B. Derrida "Shift in the velocity of a front due to a cutoff" (see the reference within the paper), 
where such a behavior is established for a similar branching-selection particle system, using both numerical simulations and heuristic arguments.
The case where $1/2 \leq p < 1$ is briefly discussed.
\end{abstract}

\maketitle

\section{Introduction}

In~\cite{BruDer1, BruDer2}, E. Brunet and B. Derrida studied a branching-selection particle system on $\Z$ enjoying the following property: as 
the number $N$ of particles in the system goes to infinity, the asymptotic speed of the population of particles in the system converges to its limiting value 
at a surprisingly slow rate, of order $\log(N)^{-2}$. 
This behavior was established both by direct numerical simulation of the particle system, and by mathematically non-rigorous arguments of the following type: 
having a finite population of $N$ particles instead of an infinite number of particles should be more or less equivalent, 
as far as the asymptotic speed is concerned, to introducing a cutoff value of $\epsilon=1/N$, in the deterministic equations that govern the time-evolution of the distribution of particles in 
the infinite-population limit.
In turn, these equations can be viewed as discrete versions of the well-known F-KPP equations, and the initial problem is thus related to that of assessing the effect 
of introducing a small cutoff in F-KPP equations, upon the speed of the travelling wave solutions.  In turn, this problem was studied by heuristic arguments 
and computer simulations (see~\cite{BruDer1, BruDer2}), and rigorous mathematical results for this last problem have recently been obtained, 
cite~\cite{BenDep1, BenDep2, DumPopKap}. 
Another approach is based on adding a small white-noise term with scale parameter $\sigma^{2}=1/N$ in the Fisher-KPP equation, see~\cite{BruDer3}, 
and rigorous results have been derived for this model too, see~\cite{ConDoe, MueMytQua}.
However, to our knowledge, no rigorous results dealing directly with a branching-selection particle system such as the one originally studied 
by Brunet and Derrida, are available.  
In this paper, we consider a branching-selection particle system that is similar (but not exactly identical) 
to the one considered by Brunet and Derrida in~\cite{BruDer1, BruDer2}. 
To be specific, we consider a particle system with $N$ particles on $\Z$ defined through the repeated application of branching and selection steps defined as follows:
\begin{itemize}
\item Branching: each of the $N$ particles is replaced by two new particles, whose positions are shifted from that of the original particle by independently performing two random walk steps according to the distribution $p \delta_{1} + (1-p) \delta_{0}$, from 
the location of the original particle;
\item Selection:  only the $N$ rightmost particles are kept among the $2N$ obtained at the branching step, to form the new population of $N$ particles.
\end{itemize}

In Section~\ref{s:elementary}, it is proved that, after a large number of iterated branching-selection steps, 
the displacement of the whole population of $N$ particles is ballistic, with deterministic asymptotic speed $v_{N}(p)$, 
and that, as $N$ goes to infinity, $v_{N}(p)$ increases to a finite limit $v_{\infty}(p)$ (which admits a more explicit characterization).  
The main results concerning the branching-selection particle system are contained in the following two theorems:
\begin{theorem}\label{t:BD-upper}
For every $0<p<1/2$, there exists $0<C_{*}(p)<+\infty$ such that, for all large $N$, 
\begin{equation}\label{e:upper-bound}  v_{\infty}(p) - v_{N}(p) \geq C_{*}(p) \log (N)^{-2} \end{equation}
\end{theorem}
\begin{theorem}\label{t:BD-lower}
For every $0<p<1/2$,   there exists $0<C^{*}(p)<+\infty$ such that, for all large $N$, 
\begin{equation}\label{e:lower-bound}  v_{\infty}(p) - v_{N}(p) \leq C^{*}(p) \log (N)^{-2} \end{equation}
\end{theorem}

In the case $p=1/2$ or $1/2<p<1$, the behavior turns out to be quite different, as stated in the following theorems.

\begin{theorem}\label{t:un-demi}
For $p=1/2$, there exists $0<C_{*}(1/2) \leq C^{*}(1/2)<+\infty$ such that, for all large $N$, 
\begin{equation}\label{e:un-demi}  C_{*}(1/2) N^{-1} \leq v_{\infty}(p) - v_{N}(p) \leq C^{*}(1/2) N^{-1}.\end{equation}
\end{theorem}

\begin{theorem}\label{t:sup-un-demi}
For $p>1/2$, there exists $0<C_{*}(p) \leq C^{*}(p)<+\infty$ such that, for all large $N$, 
\begin{equation}\label{e:sup-un-demi}  C_{*}(p) N \leq  - \log( v_{\infty}(p) - v_{N}(p)) \leq C^{*}(p) N.\end{equation}
\end{theorem}

The rest of the paper is organized as follows. In Section~\ref{s:not-def}, we provide the precise notations and definitions that are needed 
in the sequel. Section~\ref{s:elementary} contains a discussion of various elementary properties of the model we consider. 
Section~\ref{s:upper} contains the proof of Theorem~\ref{t:BD-upper}, while Section~\ref{s:lower} contains the proof of Theorem~\ref{t:BD-lower}.
The proofs of Theorems~\ref{t:un-demi} and~\ref{t:sup-un-demi} are sketched in Section~\ref{s:autres-cas}. 
Section~\ref{s:remarks} contains some concluding remarks.
\section{Notations and definitions}\label{s:not-def}
Throughout the paper, $p$ denotes a fixed parameter in $]0,1/2[$. Since the particles we consider carry no other information than their position, it is convenient to represent finite populations of particles by finite counting measures on 
$\Z$. For all $N \geq 1$, let $\C_{N}$ denote the set of finite counting measures on $\Z$ with total mass equal to $N$, and $\C$ the set of all finite counting measures on $\Z$.

For $\nu \in \C$, the total mass of $\nu$ (i.e. the number of particles in the population it describes) is denoted by $M(\nu)$. We 
denote by $\max \nu$ and $\min \nu$ respectively the maximum and minimum of the (finite) support of $\mu$.  
We also define the diameter $d(\nu):=\max \nu - \min \nu$. 
Given two positive measures $\mu, \nu$ on $\Z$, we use the notation $\mu \leq \nu$ to denote the fact that $\mu(x) \leq \nu(x)$ for every $x \in \Z$.
On the other hand, we use the notation $\prec$ to denote the stochastic order between positive measures: $\mu \prec \nu$ if and only if
$\mu([x,+\infty[) \leq \nu([x,+\infty[)$ for all $x \in \Z$. In particular, $\mu \prec \nu$ implies that $M(\mu) \leq M(\nu)$, and it is easily seen that, 
if $\mu=\sum_{i=1}^{M(\mu)} \delta_{x_{i}}$ and $\nu=\sum_{i=1}^{M(\nu)} \delta_{y_{i}}$, with $x_{1}\geq \cdots \geq x_{M(\mu)}$ and
$y_{1} \geq \cdots \geq y_{M(\nu)}$, $\mu \prec \nu$ is equivalent to $x_{i} \leq y_{i}$ for all $1 \leq i \leq M(\nu)$. From the order $\prec$ 
on $\C$, we define the corresponding stochastic order on probability measures on $\C$ and denote it by $\prec \prec$:  given two probability measures $Q$ and $R$ on 
$\C$, $Q \prec \prec R$ means that for every bounded non-decreasing function $f \ : \ (\C, \prec) \to \R$, $Q(f) \leq R(f)$. An equivalent definition is that there exists
a pair of random variables $(X,Y)$, $X$ and $Y$ taking values in $\C$, such that $X \leadsto Q$, $Y \leadsto R$, and $X \prec Y$ with probability one.

In this context, the dynamics of our particle systems can be defined through the following probability kernels. Let us first define the branching kernel $p^{Br.}_{N}$ on 
$\C_{N} \times \C_{2N}$ as follows. Given $\nu = \sum_{i=1}^{N} \delta_{x_{i}} \in \C_{N}$,  $p^{Br.}_{N}(\nu,\cdot)$ is the probability distribution 
of $\sum_{i=1}^{N} \delta_{x_{i}+Y_{1,i}} + \delta_{x_{i}+Y_{2,i}} \in \C_{2N}$, where $(Y_{\ell,i})_{1 \leq i \leq N, \ \ell=1,2}$ is a family 
of i.i.d. Bernoulli random variables with common distribution $p \delta_{1}+(1-p) \delta_{0}$.
Then, we define the selection kernel $p^{Sel.}_{N}$ on 
$\C_{2N} \times \C_{N}$ as follows. Starting from  $\nu = \sum_{i=1}^{2N} \delta_{x_{i}} \in \C_{2N}$, where $x_{1} \geq \cdots \geq x_{2N}$, 
  $p^{Sel.}_{N}(\nu,\cdot)$ is the Dirac distribution concentrated on the counting measure $\sum_{i=1}^{N} \delta_{x_{i}}$.

The kernel $p_{N}$ on $\C_{N} \times \C_{N}$ governing the evolution of particle systems with $N$ particles is then defined as the product kernel 
$p_{N} := p^{Br.}_{N} p^{Sel.}_{N}$.

In the sequel, we use the notation $(X^{N}_{n})_{n \geq 0}$ to denote a Markov chain on $\C_{N}$ whose transition probabilities are given by $p_{N}$, and which starts
at $X^{N}_{0}:=N \delta_{0}$. We assume this Markov chain is defined on a reference probability space denoted by $(\Omega,\F,\P)$.

Let $\sim_{N}$ denote the equivalence relation on $\C_{N}$ defined by $\nu \sim_{N} \mu$ if and only if there exists $m \in \Z$ such that
$\nu$ is the image measure of $\mu$ by the translation of $\Z$ $x \mapsto x+m$. 
Let $\pi_{N}$ denote the canonical projection from $\C_{N}$ to $\C_{N}/\sim_{N}$.

\section{Elementary properties of the model}\label{s:elementary}

\begin{prop}\label{p:kernel-monot}
For all $1 \leq N_{1} \leq N_{2}$, and $\mu \in \C_{N_{1}}$ and  $\nu \in \C_{N_{2}}$ such that $\mu \prec \nu$, 
 $ p_{N_{1}}(\mu,\cdot) \prec \prec  p_{N_{2}}(\nu,\cdot)$.
\end{prop}

\begin{proof}
The proof is by coupling. Let
$\mu=\sum_{i=1}^{N_{1}} \delta_{x_{i}}$ and $\nu=\sum_{i=1}^{N_{2}} \delta_{y_{i}}$, with $x_{1}\geq \cdots \geq x_{N_{1}}$ and
$y_{1} \geq \cdots \geq y_{N_{2}}$ such that $x_{i} \leq y_{i}$ for all $1 \leq i \leq N_{1}$. Let 
$(Y_{\ell,i})_{1 \leq i \leq N_{2}, \ \ell=1,2}$ denote a family 
of i.i.d. Bernoulli random variables with common distribution $p \delta_{1}+(1-p) \delta_{0}$. 
By definition, the couting measure defined by $\mu_{Br.}:=\sum_{i=1}^{2N_{1}} \delta_{x_{i}+Y_{1,i}}+ \delta_{x_{i}+Y_{2,i}}$ has the distribution $\mu p^{Br.}_{N_{1}}$, and 
$\nu_{Br.}:=\sum_{i=1}^{2N_{2}} \delta_{y_{i}+Y_{1,i}}+ \delta_{y_{i}+Y_{2,i}}$ has the distribution $\mu p^{Br.}_{N_{2}}$. It is easily checked that 
$\mu_{Br.} \prec \nu_{Br.}$, owing to the fact that $x_{i}+Y_{\ell,i} \leq  y_{i}+Y_{\ell,i}$ for all $1 \leq i \leq N_{1}$ and $\ell=1,2$.
We deduce that $\mu p^{Br.}_{N_{1}} \prec \prec \mu p^{Br.}_{N_{2}}$. Then, it is  obvious from the definition that $p_{N_{1}}^{Sel.}$ and
$p_{N_{2}}^{Sel.}$ preserve $\prec \prec$.
\end{proof}

\begin{prop}\label{p:diam-borne}
For all $N \geq 1$, and all $n \geq 0$, $ d(X^{N}_{n})  \leq \ceil{\frac{\log(N)}{\log(2)}}+1$ with probability one. 
\end{prop}

\begin{proof}

Let $m:=\ceil{\frac{\log(N)}{\log(2)}}+1$. The result is obvious for $n=0,\ldots, m$, since we perform 0 or 1 random walk steps starting from
an initial condition where all particles are at the origin. Now consider $n > m$, and let $y=\max X^{N}_{n-m}$. 
Assume first that $\min X^{N}_{k} < y$ for all $n+1-m \leq k \leq n$. Since all the random walk steps that are performed during branching steps are $\geq 0$, this implies
that all the particles descended by branching from a particle located at $y$ at time $n-m$, are preserved by the selection steps performed from $X^{N}_{n-m}$ to $X^{N}_{n}$.
Since there are $2^{m}>N$ such particles, this is a contradiction. As a consequence, we know that there must be an index  $n+1-m \leq k \leq n$ such that 
$\min X^{N}_{k} \geq y$. Since by construction $t \mapsto \min X^{N}_{t}$ is non-decreasing,  we deduce that $\min X^{N}_{n} \geq y$. 
Now, since all the random walk steps that are performed add $0$ or $1$ to the current position of a particle, we see from the definition of $y$ that $\max X^{N}_{n} \leq y+m$.  
As a consequence, $d(X^{N}_{n})=\max X^{N}_{n-m}-\min X^{N}_{n-m} \leq m$.
 \end{proof}

\begin{prop}\label{p;kernel-compat}
For all $N \geq 1$, the kernel $p_{N}$ is compatible with the canonical projection $\pi_{N}$, that is: for all $\nu_{1}, \nu_{2} \in \C_{N}$ such that 
$\nu_{1} \sim_{N} \nu_{2}$, and all
$\xi \in \C_{N} / \sim_{N}$, 
$p_{N}(\nu_{1}, \pi_{N}^{-1}(\xi)) = p_{N}(\nu_{2}, \pi_{N}^{-1}(\xi))$.
\end{prop}

\begin{proof}
Immediate: everything in the definition of the branching and selection steps is translation-invariant.
\end{proof}

\begin{prop}
For all $N \geq 1$, the sequence $\pi_{N}(X^{N}_{n})_{n \geq 0}$ is an ergodic Markov chain on a finite subset of $\C_{N} / \sim_{N}$.
\end{prop}

\begin{proof}
The fact that  $\pi_{N}(X^{N}_{n})_{n \geq 0}$ forms a Markov chain on  $\C_{N} / \sim_{N}$ is an immediate consequence of Proposition~\ref{p;kernel-compat}.
Let $\S_{N} := \{  \xi \in \C_{N} / \sim_{N}; \ \exists n \geq 0, \ \P(\pi_{N}(X^{N}_{n}) = \xi  ) > 0 \}$. 
From Proposition~\ref{p:diam-borne}, we see that $\S_{N}$ is in fact a finite set. On the other hand, given any  $\xi \in \S_{N}$, it is quite easy to
find a finite path in $\S_{N}$ starting at $\xi$ and  ending at $\pi_{N}(N \delta_{0})$ that has positive probability. As a consequence,  
the restriction of $\pi_{N}(X^{N}_{n})_{n \geq 0}$ to $\S_{N}$ is
an irreducible Markov chain. As for aperiodicity, the transition $\pi_{N}(N \delta_{0}) \mapsto  \pi_{N}(N \delta_{0})$ has e.g. a positive probability.
\end{proof}

\begin{coroll}
There exists $0<v_{N}(p)<+\infty$ such that, with probability one, and in $L^{1}(\P)$,
$$\lim_{n \to +\infty} n^{-1} \min X^{N}_{n} =  \lim_{n \to +\infty} n^{-1} \max X^{N}_{n} = v_{N}(p).$$ 
\end{coroll}

\begin{proof}
Note that, in view of Proposition~\ref{p:diam-borne}, if the two limits in the above statement exist, they must be equal. Then observe that, for all $n\geq 0$, 
conditionally upon $X^{N}_{n}$, the distributions of the increments $\max X^{N}_{n+1} - \max X^{N}_{n}$ and 
  $\min X^{N}_{n+1} - \min X^{N}_{n}$ depend only on $\pi_{N}(X^{N}_{n})$. 
    The result then follows by a classical argument using the law of large numbers for additive functionals of ergodic Markov chains.
\end{proof}

\begin{prop}
The sequence $(v_{N}(p))_{N \geq 1}$ is non-decreasing.
\end{prop}

\begin{proof}
Consequence of the fact that, when $N_{1} \leq N_{2}$, $N_{1} \delta_{0} \prec N_{2} \delta_{0}$, and of the monotonicity property~\ref{p:kernel-monot}.
\end{proof}

We can deduce from the above proposition that there exists $0<v_{\infty}(p) <+\infty$ such that
$\lim_{N \to +\infty} v_{N}(p) = v_{\infty}(p)$. A consequence of the proofs of Theorems~\ref{t:BD-upper} and~\ref{t:BD-lower} below is that 
$v_{\infty}(p)$ is in fact equal to the number $v(p)$ characterized as the unique root of the equation  $\Lambda(x)=\log(2)$, where $x \in [0,1]$ is the unknown, and where  
$\Lambda$ is the large deviations rate function associated with sums of  i.i.d Bernoulli($p$) random variables, i.e.
$\Lambda(x):=x \log(x/p) + (1-x) \log (\textstyle{\frac{1-x}{1-p}})$ for $x \in [0,1]$. We note for future use that $p<v(p)<1$ since $p \in ]0,1/2[$.

\section{The upper bound}\label{s:upper}

The essential arguments used here in the proof of the upper bound, are largely borrowed from the paper~\cite{Pem} by R. Pemantle, which deals with the closely related question 
of obtaining complexity bounds for algorithms that seek near optimal paths in branching random walks. In fact, the proof of Theorem~\ref{t:BD-upper} given below 
is basically a rewriting of the proof of the lower complexity bound in~\cite{Pem} in the special case of algorithms that do not jump, with the slight 
difference that we are dealing with $N$ independent branching random walks being explored in parallel, rather than a single branching random walk.

To explain the connexion between our model and the branching random walk, consider the following model.
Let $\BRW_{1},\ldots, \BRW_{N}$ denote $N$ independent branching random walks, each with value zero at the root, deterministic binary branching, and i.i.d. displacements
with common distribution $p \delta_{1} + (1-p) \delta_{0}$ along each edge. For $1 \leq i \leq N$, and $n \geq 0$, let $\BRW_{i}(n)$ denote the set of vertices of $\BRW_{i}$
located at depth $n$ in the tree, and let $T_{n}:=\BRW_{1}(n) \cup \cdots \cup \BRW_{N}(n)$. For every $n$, fix an a priori 
(i.e. depending only on the tree structure, not on the random walk values) total order on $T_{n}$.
We now define by induction a sequence $(G_{n})_{n \geq 0}$ of subsets such that, for each $n \geq 0$, $G_{n}$ is a random subset of  
$T_{n}$ with exactly $N$ elements. 
 First, let us set $G_{0}:= T_{0}$. 
Then, given $n \geq 0$ and $G_{n}$, let $H_{n}$ denote the subset of  $T(n+1)$ formed by the children of the vertices in $G_{n}$.
Then, define $G_{n+1}$ as the subset of $H_{n}$ formed by the $N$ vertices that are associated with the largest values of the underlying random walk (breaking ties by using the a priori order on $T_{n}$). 
It is now quite obvious that, for every $n \geq 0$, the (random) empirical distribution of the $N$ random walk values associated with the vertices in $G_{n}$ has 
the same distribution as $X^{N}_{n}$. 

Given a branching random walk $\BRW$ of the type defined above, and one of its vertices $u$, we use the notation $Z(u)$ to denote the value of the random walk associated with $u$.
The following definition is adapted from~\cite{Pem}. 
Given $0 <v < 1$ and $m \geq 1$, we say that a vertex $u \in \BRW$ is $(m,v)-$good if there is a descending path $u=:u_{0},u_{1},\ldots, u_{m}$ such that $Z(u_{i}) -Z(u_{0}) \geq v i$ for all $i \in \ig 0,m\id$.
The importance of this definition comes from the two following lemmas, adapted from~\cite{Pem}.

\begin{lemma}\label{l:beaucoup-de-bien}(Lemma 5.2 in~\cite{Pem})
Let $0<v_{1}<v_{2}<1$.
If there exists a vertex $u \in \BRW$ at depth $n$ such that $Z(u) \geq v_{2} n$, then the path from the root to $u$ 
must contain at least  $\textstyle{ \frac{v_{2}-v_{1}}{1-v_{1}}}\frac{n}{m} -1/(1-v_{1})$  vertices that are $(m,v_{1})$-good.
\end{lemma}

\begin{lemma}\label{l:borne-bien}(Proposition 2.6 in~\cite{Pem})
There is a constant $\psi>0$ such that, given a branching random walk $\BRW$, the probability that the root is $(m, v(p)-m^{-2/3})-$good is less than $\exp(-\psi m^{1/3})$.
\end{lemma}

Since Lemma~\ref{l:beaucoup-de-bien} admits so short a proof, we reproduce it below for the sake of completeness. On the other hand, to give a very rough idea where the 
$\exp(-\psi m^{1/3})$ in Lemma~\ref{l:borne-bien} comes from, let us just mention 
that it corresponds to the probability that a random walk remains confined in a tube of size $m^{1/3}$ around its mean, for $m$ time steps. Dividing the $m$ steps
into $m^{1/3}$ intervals of size $m^{2/3}$, we see that this amounts to asking for the realization of $m^{1/3}$ independent events, each of which has a probability of order a 
constant, by the usual Brownian scaling.

\begin{proof}[Proof of Lemma~\ref{l:beaucoup-de-bien}](From~\cite{Pem}.)
Consider a vertex $u$ as in the statement of the lemma. 
Consider the descending path $root=:x_{0},\ldots, x_{n}:=u$ from the root to $u$. 
Let then $\tau_{0}:=0$, and, given $\tau_{i}$, define inductively $\tau_{i+1}:=\inf \{ j \geq \tau_{i}+1; \  Z(x_{j}) < Z(x_{\tau_{i}})+v_{1} (j-\tau_{i})
 \mbox{ or } j=\tau_{i}+m  \}.$ Now color $x_{0},\ldots, x_{n-1}$ according to the following rules: 
 if $Z(x_{\tau_{i+1}}) \geq Z(x_{\tau_{i}})+v_{1} (j-\tau_{i})$ and $\tau_{i+1} \leq n+1$, then $x_{\tau_{i}},\ldots,x_{\tau_{i+1}-1}$ are colored 
 red. Note that this yields a segment of $m$ red vertices, and that $x_{\tau_{i}}$ is then $(m,v_{1})-$good. Otherwise,  $x_{\tau_{i}},\ldots,x_{\tau_{i+1}-1}$ are colored blue. Let $V_{red}$ (resp. $V_{blue}$) denote the number of red (resp. blue) vertices
 in $x_{0},\ldots, x_{n-1}$. 
 Then decompose $Z(u)$ into the contributions of the red and blue vertices. On the one hand, the contribution of red vertices is $\leq V_{red}$. On the other, 
 the contribution of blue vertices is $\leq V_{blue} \times v_{1}+m$, where the $m$ is added to take into account a possible last segment colored in blue only
 because it has reached depth $n$. Writing that $n=V_{red}+V_{blue}$, we deduce that 
 $v_{2} n \leq V_{red} + v_{1}(n-V_{red}) + m$, so that $V_{red} \geq \textstyle{ \frac{v_{2}-v_{1}}{1-v_{1}}}n -m/(1-v_{1})$.
  Then use the fact that at least $V_{red}/m$ vertices
 are $(m,v_{1})$-good.
 \end{proof} 
 
In~\cite{Pem}, Lemmas~\ref{l:beaucoup-de-bien} and~\ref{l:borne-bien} are used in combination with an elaborate second moment argument.
In the present context, the following quite simple first moment argument turns out to be sufficient. 

\begin{proof}[Proof of Theorem~\ref{t:BD-upper}]
Consider an integer $r \geq 1$, and let  $m:=r\floor{\log(N)^{3}}$. 
Let $n \geq m$, and let $B_{n}$ denote the number of vertices in $\BRW_{1} \cup \cdots \BRW_{N}$ that are $(m,v(p)-2m^{-2/3})-$good (each with respect to the BRW it belongs to)
and belong to $G_{0} \cup \cdots \cup G_{n}$. From Lemma~\ref{l:beaucoup-de-bien} and the definition of $(G_{i})_{i \geq 0}$, 
we see that the fact that at least one vertex in $G_{n}$ has a value larger than $(v(p)-m^{-2/3})n$ implies that, for large $n$ (depending on $m$),   
$B_{n} \geq  n m^{-5/3}$. On the other hand, $B_{n}$ can be written as 
\begin{equation}\label{e:decomp}B_{n}:=\sum_{u \in \BRW_{1} \cup \cdots \BRW_{N}} \un(\mbox{  $u$ is $(m,v(p)-2m^{-2/3})-$good}) \un( u \in G_{0} \cup \cdots \cup G_{n}).\end{equation}
Now observe that, by definition, for a vertex $u$ of depth $\ell$, by definition, the event $u\in G_{0} \cup \cdots \cup G_{n}$ is measurable with respect to the random walk steps
performed up to depth $\ell$, while
the event that  $u$ is $(m,v(p)-2m^{-2/3})-$good is measurable with respect to the random walk steps performed starting from depth $\geq \ell$. As a consequence, these two events are independent.
Since the total number of vertices in $G_{0} \cup \cdots \cup G_{n}$ is equal to $N (n+1)$, we deduce from Lemma~\ref{l:borne-bien} and~(\ref{e:decomp}) that
$E \left(     B_{n}    \right)  \leq  N(n+1) \exp(-\psi m^{1/3}) $. Using Markov's inequality, and letting $n$ go to infinity, we deduce that
$$\limsup_{n \to +\infty} P( B_{n} \geq   n m^{-5/3} ) \leq  N m^{5/3}\exp(-\psi m^{1/3}).$$
Now, remembering that $P(B_{n})=\P(\max X^{N}_{n} \geq (v(p) - m^{-2/3}) n)$, and using the fact that $\max X^{N}_{n} \leq n$ with probability one, 
 we finally deduce that
$$\limsup_{n \to +\infty} n^{-1}E(\max X^{N}_{n})  \leq  v(p) - m^{-2/3} + N m^{5/3}\exp(-\psi m^{1/3}).$$
Choosing $r$ large enough in the definition of $m$ makes the third term in the above r.h.s. negligible with respect to the second term, as $N$ goes to infinity. 
The conclusion follows.
\end{proof}

\section{The lower bound}\label{s:lower}

The proof of the lower bound on the convergence rate of $v_{N}(p)$ to $v_{\infty}(p)$ is in some sense a rigorous version of the heuristic argument of Brunet and Derrida 
according to which we should compare the behavior of the particle system with $N$ particles, with a version of the infinite population limit modified by a cutoff at $\epsilon=1/N$. 

Indeed, given a finite positive measure $\nu$ on $\Z$, let $F^{Br.}(\nu)$ be the measure defined by:
$F^{Br.}(\nu):= 2 \nu \star (p \delta_{1} + (1-p) \delta_{0})$. This $F^{Br.}$ describes the evolution of the distribution of particles in the infinite
population limit above the threshold imposed by the selection step. The idea of the proof of the lower bound is to control the discrepancy between the finite and infinite
population models above this threshold. One important observation is  that, to prove a lower bound, one does not necessarily 
have to control the number of particles at every site, but may focus instead on sites where the probability of finding a particle is not too small. 

We note the following two immediate properties of $F^{Br.}$: if $\mu \leq \nu$, then $F^{Br.}(\mu) \leq F^{Br.}(\nu)$, and, 
if $g \in \R_{+}$, $F^{Br.}(g \nu) = g F^{Br.}(\nu)$.

\subsection{Admissible sequences of measures}\label{s:admiss}

Throughout this section, we consider $\epsilon>0$, 
$0<\alpha<v(p)$,  $\beta>1$, and $m \geq q:=\ceil{\textstyle{\frac{v(p)}{1-v(p)}}}$.

We say that a (deterministic) sequence $\delta_{0}=:\nu_{1},\ldots,\nu_{m}$ of positive measures on $\Z$ with finite
 support, is $(\epsilon,\alpha, \beta)-$admissible, if the following properties hold:
\begin{itemize}
\item[(i)] $\nu_{i}=(2p)^{i}\delta_{i}$ for $0 \leq i \leq q$;
\item[(ii)] for all $q+1 \leq i \leq m$, $\nu_{i} \leq F^{Br.}(\nu_{i-1})$;
\item[(iii)] for all $0 \leq i \leq m-1$, and all $x \in \Z$ such that $\nu_{i}(x) > 0$, $\nu_{i}(x) \geq \epsilon$;
\item[(iv)] for all $q \leq i \leq m-1$, the support of $\nu_{i}$ is contained in the interval $[(v(p)-\alpha)(i+1), +\infty[$;
\item[(v)] $\nu_{m}(\Z) \geq \beta+1$.
\end{itemize}
Note that the definition of $q$ makes property (iv) automatic for $i=q$.

Let $B:=\{  \min(X^{N}_{i}) <     (v(p)-\alpha) i     \mbox{ for all } 1 \leq i \leq m  \}$. The interest of admissible sequences of measures lies in the possibility 
of bounding $\P(B)$ from above, as explained in the following lemma.
\begin{lemma}\label{l:utilisation-admissible}
Consider an $(\epsilon,\alpha,\beta)-$admissible sequence $\nu_{0},\ldots,\nu_{m}$. 
Let $K:=\sum_{i=0}^{m-1} \# \mbox{supp}(\nu_{i})$, and $\delta:=1-\exp\left(-\textstyle{\frac{\log(\beta)}{m}}\right)$. Then the following inequality holds:
$$ \P(B) \leq 
2 K \exp\left(- N \beta^{-1} \epsilon  p \delta^{2} \right).$$
\end{lemma}

Before proving the above lemma, we recall the following classical estimate for binomial random variables (see e.g.~\cite{McD}):
\begin{lemma}\label{l:dev-bino}
Let $n \geq 1$ and $0<r<1$, and let $Z$ be a binomial$(n,r)$ random variable. Then, for all $0<\delta<1$, the probability that $Z \leq (1-\delta) nr$ is less than
 $\exp(-\textstyle{\frac{1}{2}}nr\delta^{2})$.
\end{lemma}

\begin{proof}[Proof of Lemma~\ref{l:utilisation-admissible}]

For $k \in \ig 0, m \id$,  let 
$A_{k}:=  \bigcap_{n \in \ig  1,k \id} \{ X^{N}_{n} \geq N (1-\delta)^{n} \nu_{n} \}$. Note that $A_{0}$ is the certain event.
For $0 \leq k \leq m-1$, and $x \in \Z$, define $N^{1}_{k}(x)$ (resp. $N^{0}_{k}(x)$) to be the number of particles that are created  from a particle 
at position $x$ in $X^{N}_{k}$ during the branching step
applied to $X^{N}_{k}$, and that have a position equal to $x+1$ (resp. $x$). By definition, conditional upon 
  $X^{N}_{0},\ldots, X^{N}_{k}$, $N^{1}_{k}(x)$  (resp. $N^{0}_{k}(x)$) follows a  binomial distribution with parameters $(2X^{N}_{k}(x),p)$ (resp. $(2X^{N}_{k}(x),1-p)$).
For $k \in \ig 0, m-1 \id$, let $C_{k} := \{  N^{\ell}_{k}(x) \geq (1-\delta) 2 X^{N}_{k}(x) p ; \ x \in \mbox{supp}(\nu_{k})  , \ \ell=0,1  \}$.
Note that,  by definition of $\delta$, $(1-\delta)^{k} \geq \beta^{-1}$ for all $0 \leq k \leq m$.
In view of condition (iii), we deduce that, for $k \in \ig 1, m$, on  $A_{k-1}$,  $X^{N}_{k-1}(x) \geq N  \beta^{-1} \epsilon$ for all $x \in \mbox{supp}(\nu_{k-1})$.
As a consequence,  Lemma~\ref{l:dev-bino} yields the fact that 
\begin{equation}\label{e:borne-bino-branch}\P(A_{k-1} \cap C_{k-1}^{c}) \leq  2 \#\mbox{supp}(\nu_{k-1})   \exp\left(- N  \beta^{-1}  \epsilon p  \delta^{2} \right),\end{equation}
where we have used the union bound and the fact that $p \leq 1-p$ to combine the bounds given by Lemma~\ref{l:dev-bino} for all the $N^{\ell}_{k-1}(x)$, with $\ell \in \{ 0,1\}$, 
and $x \in \mbox{supp}(\nu_{k-1})$.
Now consider $k \in \ig 1, q \id$. On $B$, all the particles counted by  $N^{1}_{k-1}(k-1)$  must be kept after the selection step leading to $X^{N}_{k}$, since 
$k \geq (v(p)-\alpha)k$.
As a consequence, $X^{N}_{k}(k) \geq  N^{1}_{k-1}(k-1)$, and  we deduce that 
\begin{equation}\label{e:inclus-1} A_{k-1} \cap C_{k-1} \cap B \subset A_{k}.\end{equation}
Assume now that $k \in \ig q+1,m \id$. If $B$ holds, we see that, according to (iv),  for all $x$ in the support of $\nu_{k-1}$, the particles counted by 
$N^{1}_{k-1}(x)$ and $N^{0}_{k-1}(x)$ are all kept after the selection step leading to $X^{N}_{k}$. 
As a consequence, on $C_{k-1}$,  $X^{N}_{k} \geq (1-\delta) F^{Br.}(X^{N}_{k-1} \un(\mbox{supp}(\nu_{k-1})))$, so that, 
on $B \cap A_{k-1} \cap C_{k-1}$,  $X^{N}_{k} \geq (1-\delta)^{k} \nu_{k}$, since $X^{N}_{k-1} \geq (1-\delta)^{k-1} \nu_{k-1}$
and $\nu_{k} \leq F^{Br.}(\nu_{k-1})$ by assumption (ii). We deduce that 
\begin{equation}\label{e:inclus-2}A_{k-1} \cap C_{k-1} \cap B \subset A_{k}.\end{equation}
Now observe that, on $A_{m}$, one must have $X^{N}_{m}(\Z) \geq N (1-\delta)^{m} \nu_{m}(\Z) \geq N  \beta^{-1} \nu_{m}(\Z) > N$, a contradiction, so that
$A_{m} = \emptyset$.
From~(\ref{e:inclus-1}), ~(\ref{e:inclus-2}), we deduce that 
$\P(B) \leq \sum_{k=0}^{m-1} \P(A_{k-1} \cap C_{k-1}^{c})$, and, using~(\ref{e:borne-bino-branch}), we deduce the result.
\end{proof}

Let us now relate the above results with estimates on $v_{N}(p)$.
Define the random variable
$L:= \inf \{  1 \leq i \leq m ; \  \min(X^{N}_{i}) \geq     (v(p)-\alpha) i   \}$, with the convention that $\inf \emptyset := m$.

\begin{prop}\label{p:borne-inf-B}
For all $0<p<1/2$, for all $N \geq 1$, 
$$v_{N}(p) \geq (v(p)-\alpha) (1-m\P(B)).$$
\end{prop}

\begin{proof}
We define a modified branching-selection process $(Y^{N}_{n})_{n \geq 0}$, composed of a succession of runs.
Start with $L_{0}:=0$ and $H_{0}:=0$, and $i:=0$, and do the following.
\begin{itemize}
\item[1)] Set $Z^{N}_{0}:=N \delta_{H_{i}}$ and $k:=0$.
\item [2)]Do the following:
\item[] $\{$ let $k:=k+1$ and generate $Z^{N}_{k}$ from the distribution $p_{N}(Z^{N}_{k-1}, \cdot)$. $\}$
\item[3)] Return to 2) until $k=m$ or $\min Z^{N}_{k} \geq (v(p)-\alpha) k + H_{i}$. 
\item[4)] Set $L_{i+1}:=L_{i}+k$ and $H_{i+1}:=\min Z^{N}_{k}$
\item[5)] Set $(Y^{N}_{L_{i}},\ldots,Y^{N}_{L_{i+1}-1}):=(Z^{N}_{0},\ldots,Z^{N}_{k-1})$
\item[6)] Let $i:=i+1$ and return to 1) for the next run. 
\end{itemize}
One may describe the above process as follows: starting from a reference position $H_{i}$ and a reference time index $L_{i}$, a run behaves like the original branching-selection process until
either $m$ steps have been performed or the minimum value in the population of particles exceeds the reference position by an amount of at least $(v(p)-\alpha)$ times the number of steps
performed since the beginning of the run. Then the current population of $N$ particles is collapsed onto its minimum position,  the reference position is updated 
to this minimum position, and the time index to the current time, 
and a new run is started.
Our modified process has a natural regeneration structure yielding the fact that the sequences $(L_{i+1}-L_{i})_{i \geq 0}$ and $(H_{i+1}-H_{i})$ are i.i.d.
The common distribution of the $L_{i+1}-L_{i}$ is that of $L$, while the common distribution of the $H_{i+1}-H_{i}$ is that of 
$\min X^{N}_{L}$. From the  fact that $\min Y^{N}_{L_{i}} \geq H_{i}$,  it is easy to deduce that, with probability one, $\lim_{n \to +\infty} n^{-1}\min Y^{N}_{n} = \textstyle{\frac{\E(\min X^{N}_{L})}{\E(L)}}$, and, 
since $0 \leq n^{-1}Y^{N}_{n} \leq 1$, we also have that $\lim_{n \to +\infty} n^{-1}\E(\min Y^{N}_{n}) = \textstyle{\frac{\E(\min X^{N}_{L})}{\E(L)}}$.
Now, by definition, $\min X^{N}_{L} \geq (v(p)-\alpha) L \un(B^{c})$, so that 
$\E(\min X^{N}_{L}) \geq (v(p)-\alpha) (\E(L) - \E(L \un(B)))$.
Using the fact that $1 \leq L \leq m$, we obtain that $\textstyle{\frac{\E(\min X^{N}_{L})}{\E(L)}} \geq (v(p)-\alpha) (1-m\P(B))$.

Now, it should be intuitively obvious that the modified process $(Y^{N}_{n})_{n \geq 0}$ is in some sense a lower bound for the original process $(X^{N}_{n})_{n \geq 0}$, 
since we modify the original dynamics in a way that can only lower positions of particles. It is in fact an easy consequence of Proposition~\ref{p:kernel-monot} that, for all $n$,
the distribution of $Y_{n}^{N}$ is stochastically dominated by that of $X_{n}^{N}$. A consequence is that $\E(\min Y^{N}_{n}) \leq \E(\min X^{N}_{n})$.The result follows.
\end{proof}

\subsection{Construction of an admissible sequence of measures}
Let $A$ denote an integer $\geq 4$, and $m \geq q$. Then let $a_{m}:=\floor{m^{1/3}}$, $c_{m}:=\floor{m^{2/3}}$, 
$s_{m}:=\floor{\textstyle{\frac{a_{m}}{2(1-v(p))}}}$.
Define $d_{m}$ by $d_{m}(k):=k$ for $k \in \ig 1 ,  s_{m} \id$, and $d_{m}(k):= v(p) k + a_{m}$ for $k \in \ig  s_{m} +1 , m \id$.
Define $g_{m}$ by $g_{m}(k):=k$ for $k \in \ig 1 ,  s_{m} \id$, $g_{m}(k):= v(p) (k+1)$ for $k \in \ig  s_{m}+1 , m-c_{m} \id$. 
and $g_{m}(k):= v(p) k - A a_{m}$ for $k \in \ig m-c_{m}+1,m \id$. 

Then define a sequence of measures $(\nu_{i})_{i \in \ig 0 , m \id}$ on $\Z$ as follows.
Let $(S_{i})_{i \in \ig 0, m \id}$ denote a simple random walk on $\Z$ starting at zero, with step distribution $p \delta_{1} + (1-p) \delta_{0}$, governed by a probability measure $P$, 
then let
$$\nu_{i}(x):=2^{i} P\left[ g_{m}(k) \leq   S_{k} \leq d_{m}(k) \mbox{ for all } \ k \in \ig 0, i\id , \, S_{i}=x\right].$$

The main result in this section is the following:
\begin{prop}\label{p:choix-de-seq-admissible}
For large enough $A$, there exists $\chi(A)>0$ such that, for all large enough $m$, the above sequence is $(\exp(-\chi(A)m^{1/3}, 2A m^{-2/3}, 2008) $-admissible.
\end{prop}

We need to establish several results  before we can prove the above proposition.

First, consider the modified probability measure $\hat{P}$ defined by
$$\frac{d\hat{P}}{dP} := \left(\frac{v(p)}{p}\right)^{S_{m}}  \left(\frac{1-v(p)}{1-p}\right)^{m-S_{m}}.$$
With respect to $\hat{P}$,  $(S_{i})_{i \in \ig 0, m \id}$ is a simple random walk on $\Z$ starting at zero, with step distribution 
$v(p) \delta_{1} + (1-v(p)) \delta_{0}$. 

We now rewrite $\nu_{i}(x)$ in terms of this change of measure. To this end,  introduce the compensated random walk defined by
$\hat{S}_{i}:=S_{i}-v(p) i$, let also
$\hat{g}_{m}(k):=g_{m}(k) - v(p) k $ and $\hat{d}_{m}(k):=d_{m}(k) - v(p) k$. 
 Finally, let $\gamma:=\frac{p/(1-p)}{v(p)/(1-v(p))}$, and note that $\gamma<1$ since $p<v(p)$. 
After a little algebra involving the  definition of $v(p)$ in terms of $\Lambda$, we obtain the following expression:
\begin{equation}\label{e:change-of-measure-1}\nu_{i}(x) = 
 \hat{E} \left[  \gamma^{\hat{S}_{i}}   
 \un \left(\hat{g}_{m}(k) \leq   \hat{S}_{k} \leq \hat{d}_{m}(k) \mbox{ for all } \ k \in \ig 0, i\id , \, S_{i}=x \right) \right],\end{equation}
where $\hat{E}$ denotes expectation with respect to $\hat{P}$.
From the definition of $\hat{g}_{m}$ and $\hat{d}_{m}$, we see that, for large $m$, the only values of $\hat{S_{i}}$ that contribute in the above expectation are 
$\leq m^{1/3}$. As a consequence, 
\begin{equation}\label{e:change-of-measure-2}
\nu_{i}(x) \geq   \gamma^{ m^{1/3} }   
\hat{P} \left[ \hat{g}_{m}(k) \leq   \hat{S}_{k} \leq \hat{d}_{m}(k) \mbox{ for all } \ k \in \ig 0, i\id , \, S_{i}=x  \right],\end{equation}

\begin{lemma}\label{l:proba-totale}
For some constant $\zeta_{1}>0$, as $m$ goes to infinity, 
$$\hat{P} \left[ \hat{g}_{m}(k) \leq   \hat{S}_{k} \leq \hat{d}_{m}(k) \mbox{ for all } \ k \in \ig 0, m\id  \right] 
\geq \exp(-\zeta_{1} m^{1/3}).$$
\end{lemma}

We need an elementary lemma before the proof of Lemma~\ref{l:proba-totale}.

\begin{lemma}\label{l:proba-de-rester}
Consider a random walk $(Z_{i})_{i \geq 0}$ on $\R$, defined by $Z_{i}:=Z_{0}+\varepsilon_{1} + \cdots + \varepsilon_{i}$ for $i \geq 1$. Assume that the 
increments $\varepsilon_{i}$ are i.i.d. with respect to some probability measure $Q$ and satisfy $E(\varepsilon_{1})=0$ and $0<Var(\varepsilon_{1})<+\infty$.  
Then there exists $\lambda > 0$ such that, for all $m$ large enough, on the event that $ a_{m}/3 \leq Z_{0} \leq 2a_{m}/3 $,
\begin{equation}\label{e:estimation-prob-1}Q \left[  \left. v \leq Z_{i} \leq a_{m} ;  \ i \in \ig  0, c_{m} \id , \  a_{m}/3 \leq Z_{c_{m}} \leq 2a_{m}/3  \right|  Z_{0} \right] \geq \lambda,\end{equation}
\begin{equation}\label{e:estimation-prob-2}Q \left[ \left. a_{m}/4 \leq Z_{i} \leq 3a_{m}/4 ;  \ i \in \ig  0, c_{m} \id  \right| Z_{0} \right] \geq \lambda.\end{equation}
\end{lemma}

\begin{proof}[Proof of Lemma~\ref{l:proba-de-rester}]
 
We use the convergence of the distribution of the random process $(M^{m}_{t})_{t \in [0,1]}$
defined by $M^{m}_{0}:=0$, $M^{m}_{i/c_{m}}=a_{m}^{-1}(\varepsilon_{1}+\cdots + \varepsilon_{i})$ for  $i \in \ig 1, c_{m} \id$,
and interpolated linearly on each $[i/c_{m},(i+1)/c_{m} ]$ towards the Brownian motion
 (on the space of real-valued continuous functions on  $[0,1]$ equipped with the sup norm). 
An easy consequence is that there exists $\lambda_{1}>0$ such that, for all large $m$, 
 $$Q \left[ -a_{m}/7 \leq Z_{i}-Z_{0} \leq a_{m}/7 ;  \ i \in \ig  1, c_{m} \id , \ 0 \leq Z_{c_{m}}-Z_{0} \leq a_{m}/7  \right] \geq \lambda_{1}.$$ 
This estimate proves~(\ref{e:estimation-prob-1}) when $Z_{0}$ belongs to $[a_{m}/3, a_{m}/2]$. A symmetric argument works when $Z_{0}$ belongs to
$[a_{m}/2, 2a_{m}/3]$. 
The proof of~(\ref{e:estimation-prob-2}) is quite similar.
\end{proof}

\begin{proof}[Proof of Lemma~\ref{l:proba-totale}]

First note that, for all large enough $m$, $a_{m}/3 \leq s_{m} - v(p) s_{m} \leq 2a_{m}/3$. Then, 
$$\hat{P} \left[     \hat{g}_{m}(k)\leq \hat{S}_{k} \leq   \hat{d}_{m}(k)  ; \ k \in \ig 0, s_{m}\id     \right] = v(p)^{s_{m}}.$$  

Divide the interval $\ig s_{m}+1 , m\id$ into consecutive intervals
$I_{j}$, $j \in \ig 1, h_{m}\id$, where each $I_{j}$ for $j \in \ig 1, h_{m}-1\id$ is of the form
 $I_{j}:=\ig s_{m}+1+ (j-1) c_{m} ,  s_{m}+1+j c_{m} \id$, while the last interval is 
 $I_{h_{m}}:= \ig s_{m}+1+ (h_{m}-1) c_{m}   , m \id$,  
 whose length is less than or equal to $c_{m}$.
 For $i \in  \ig 1, h_{m}\id$, let $b_{m,j-1}$ and $b_{m,j}$ be defined by $I_{j}=\ig b_{m,j-1}, b_{m,j} \id$. 
 Now, for $i \in \ig 1, h_{m}-1 \id$,  define the event 
$\Gamma_{i}:=\left\{ v \leq \hat{S}_{k} \leq a_{m} ;  \ k \in I_{i} , \  a_{m}/3 \leq Z_{b_{m,i}} \leq 2a_{m}/3  \right\}$, and let 
$\Gamma_{h_{m}}:=\left\{ a_{m}/4 \leq \hat{S}_{k} \leq 3a_{m}/4 ;  \ k \in I_{h_{m}}  \right\}$.
 
 It is easily checked that, given that $S_{s_{m}}=s_{m}$, 
 $$\bigcap_{i \in \ig 1, h_{m} \id} \Gamma_{i} \subset \bigcap_{k \in \ig s_{m}+1,m \id} 
 \{   \hat{g}_{m}(k)\leq \hat{S}_{k} \leq   \hat{d}_{m}(k) \}.$$ 
 Thanks to Lemma~\ref{l:proba-de-rester} and to the Markov property of $\hat{Z}$ with respect to $\hat{P}$, we deduce that
  $$\hat{P} \left[     \hat{g}_{m}(k)\leq \hat{S}_{k} \leq   \hat{d}_{m}(k)  ; \ k \in \ig 0,m \id      \right]  
 \geq v(p)^{s_{m}} \lambda^{h_{m}}.$$
(We use exactly Lemma~\ref{l:proba-de-rester} for intervals $I_{i}$ with $i \in \ig 1, h_{m}-1 \id$, while, for $i=h_{m}$, we use that fact that the length of 
$I_{j}$ is $\leq c_{m}$, whence the fact that the conditional probability of $\Gamma_{h_{m}}$ given $\hat{S}_{0},\ldots, \hat{S}_{b_{m,h_{m}-1}}$ is larger than or equal to the 
probability appearing in~(\ref{e:estimation-prob-2}) in Lemma~\ref{l:proba-de-rester}.)
Using the fact that $h_{m} \sim m^{1/3}$ for large $m$, the conclusion follows.
\end{proof}

\begin{lemma}\label{l:proba-eps}
There exists $\zeta_{2}(A)>0$ such that, as $m$ goes to infinity, 
$$\inf \{  \nu_{i}(x); \ i \in \ig 0, m-1\id , \ \nu_{i}(x)>0  \} \geq \exp(-\zeta_{2}(A) m^{1/3}).$$
\end{lemma}

We need the following lemma before giving the proof.

\begin{lemma}\label{l:relie}
Let $\rho, \sigma \in \R$, with $\rho +1 < \sigma$, $v \in ]0,1[$ and, 
let $\ell$ be an integer such that $\sigma + \ell v < \rho + \ell$ and $\rho + v \ell > \sigma$.
Then, for all $x \in \Z \cap [\rho, \sigma]$, and all $y \in \Z \cap [\rho + v \ell, \sigma + v \ell]$, there exists
a sequence $x=:x_{0}, x_{1},\ldots, x_{\ell}:=y$ such that $x_{i+1}-x_{i} \in \{0,1 \}$
for all $i \in \ig 0, \ell-1 \id$, and $\rho + v i \leq x_{i} \leq \sigma + vi$ for all $i \in \ig 0, \ell \id$. 
\end{lemma}

\begin{proof}[Proof of Lemma~\ref{l:relie}]
Consider $x \in \Z \cap [\rho, \sigma]$, and $y \in \Z \cap [\rho + v \ell, \sigma + v \ell]$.
Define inductively the sequence $(\tau_{i},h_{i})_{i \geq 0}$ as follows. 
Let $\tau_{0}:=x$, $h_{0}:=x$. Our assumption that $\rho +1 < \sigma$ guarantees that $x$ or $x+1$ belongs to $[\rho + v, \sigma + v]$.
If $x+1$ belongs to $[\rho + v , \sigma + v ]$, then let $d:=0$. Otherwise, let $d:=1$.
 Then, consider $i \geq 1$. If $i+d$ is even, let 
$\tau_{i}:=\max \{   j \in \ig \tau_{i-1}+1, +\infty \ig; \ h_{i-1} \geq \rho+vj \}$, and let $h_{i}:=h_{i-1}$.
If $i+d$ is odd, let $\tau_{i}:=\max \{   j \in \ig \tau_{i-1}+1, +\infty \ig; \ h_{i-1}+j-\tau_{i-1} \leq \sigma+vj \}$, 
and let $h_{i}:=h_{i-1}+\tau_{i}-\tau_{i-1}$. 

The fact that and $\rho +1 < \sigma$  and $v \in ]0,1[$  guarantees that every term in the sequence is finite.
 Define  
the sequence $(z_{k})_{k \geq 0}$ by $z_{k}:=h_{i-1}$ for $k \in \ig  \tau_{i-1}, \tau_{i} \id$ when $i+d$ is even, 
and  $z_{k}:=h_{i-1}+k-\tau_{i-1}$ for $k \in \ig  \tau_{i-1}, \tau_{i} \id$ when $i+d$ is odd.
Our assumptions yield the fact that $\rho + v k \leq z_{k} \leq \sigma + vk$ for all $k \geq 0$.
 Now note that, if  $y=z_{m}$,  the path $z_{0},\ldots, z_{m}$ solves our problem.
 If $y>z_{m}$, the assumption that $\rho+\ell > \sigma + v \ell$ plus elementary geometric considerations 
 show that there must be
some  $k \in \ig 0, m-1\id$ such that
 $y-x_{k}=\ell-k$. From the path $x_{0},\ldots,x_{k}$, add $+1$ steps and stop
at length $\ell$. The corresponding path solves the problem. Now, if $y<z_{m}$, our assumption that $\sigma < \rho + v \ell$ plus elementary geometric
considerations show that  
there must be a $k \in \ig \tau_{\kappa-1}, m-1\id$ such that $x_{k}=y$. From the path $x_{0},\ldots, x_{k}$, add $+0$ steps
 and stop at length $\ell$. The corresponding path solves the problem. 
\end{proof}

\begin{proof}[Proof of Lemma~\ref{l:proba-eps}]
In the sequel, let $x_{k}:=k$ for $k \in \ig 0,s_{m} \id$, and note that, for large enough $m$, 
$ g_{m}(k )\leq x_{k} \leq d_{m}(k)$ for all $k \in \ig 0,s_{m} \id$.

Let $\alpha>\max(v(p)^{-1}, (1-v(p))^{-1})$, and let $f_{m}:=\floor{\alpha a_{m}}$.
For $i \in \ig s_{m}+1, s_{m}+ 2 f_{m}\id$, and $x$ in the support of $\nu_{i}$, there must by definition be 
 a sequence $s_{m}=:x_{s_{m}},\ldots, x_{i}:=x$ such that 
$x_{j+1}-x_{j} \in \{0,1 \}$ and $g_{m}(j) \leq x_{j} \leq d_{m}(j)$ for all $j \in \ig s_{m} , i \id$. 
As a consequence,  for every   $i \in \ig 0, s_{m}+ 2f_{m} \id$ and $x \in \mbox{supp}(\nu_{i})$,
\begin{equation}\label{e:merci-alex-1}\hat{P} \left[ S_{k}=x_{k} ; \ k \in \ig 0, i\id , \, S_{i}=x  \right] \geq  (\min(v(p),1-v(p))^{s_{m}+2f_{m}}.\end{equation}

Now assume that $i \in \ig  s_{m}+2 f_{m}+1 ,  m-c_{m}\id$, and consider the distribution of $S_{i-f_{m}}$ with respect to $\hat{P}$, 
conditional upon $g_{m}(k) \leq   S_{k} \leq d_{m}(k) \mbox{ for all } \ k \in \ig 0, i- f_{m}\id$. This probability
distribution is concentrated on the set $\Z \cap [g_{m}(i-f_{m}), d_{m}(i- f_{m})]$, which contains at most $a_{m}$ elements.
Consequently, there exists $u \in \Z \cap [g_{m}(i-f_{m}), d_{m}(i-f_{m})]$ such that 
$$\hat{P} \left[ \left. S_{i-f_{m}} = u  \right| g_{m}(k) \leq   S_{k} \leq d_{m}(k) ;  \ k \in \ig 0, i-f_{m}\id  \right]  
\geq 1/a_{m}.$$
Now let $x$ belong to the support of $\nu_{i}$. By definition, $x \in  \Z \cap [g_{m}(i), d_{m}(i)]$. In view of the definition of $\alpha$, we see that, 
for all $m$ large enough, we can apply Lemma~\ref{l:relie} to obtain the existence of a sequence $u=:x_{i-f_{m}},\ldots, x_{i}:=x$ such that 
$x_{j+1}-x_{j} \in \{0,1 \}$ and $g_{m}(j) \leq x_{j} \leq d_{m}(j)$ for all $j \in \ig i-f_{m} , i \id$. 
Thanks to the Markov property of $(S_{k})_{k \geq 0}$ with respect to $\hat{P}$, we deduce that 
$$  \hat{P} \left[g_{m}(k) \leq   S_{k} \leq d_{m}(k) ; \ k \in \ig 0, i\id  , \ S_{i}=x  \right] $$
is larger than or equal to 
$$\hat{P} \left[g_{m}(k) \leq   S_{k} \leq d_{m}(k) , \ k \in \ig 0, i-f_{m}\id  \right]  (1/a_{m})  (\min(v(p),1-v(p))^{f_{m}}.$$
From Lemma~\ref{l:proba-totale}, we have that, for $m$ large enough,  
$$ \hat{P} \left[g_{m}(k) \leq   S_{k} \leq d_{m}(k) , \ k \in \ig 0, m \id  \right]  \geq \exp(-\zeta_{1} m^{1/3}),$$ 
so that, for every $i \in \ig  s_{m}+2 f_{m}+1 ,  m-c_{m}\id$,
\begin{eqnarray}\label{e:merci-alex-2}  \hat{P} \left[g_{m}(k) \leq   S_{k} \leq d_{m}(k) , \ k \in \ig 0, i\id  , \ S_{i}=x  \right] \geq
 \exp(- \zeta_{1} m^{1/3})   \\ \nonumber  \times  (1/a_{m}) (\min(v(p),1-v(p))^{f_{m}}.\end{eqnarray} 
For $i \in \ig  m-c_{m}+1, m-c_{m} + 2(A+1)f_{m} \id$, any $x$ in the support of $\nu_{i}(x)$ is such that there exists $u$ in the support of $\nu_{m-c_{m}}$
 and a sequence $u=:x_{m-c_{m}},\ldots, x_{i}:=x$ such that 
$x_{j+1}-x_{j} \in \{0,1 \}$ and $g_{m}(j) \leq x_{j} \leq d_{m}(j)$ for all $j \in \ig m-c_{m} , i \id$.
As a consequence, 
$$\hat{P}\left[   g_{m}(k) \leq   S_{k} \leq d_{m}(k) ; \ k \in \ig 0, i\id  , \ S_{i}=x      \right]$$
is larger than or equal to
\begin{eqnarray*}\hat{P}\left[  
 g_{m}(k) \leq   S_{k} \leq d_{m}(k) ; \ k \in \ig 0, m-c_{m}\id  , \ S_{m-c_{m}}=u      \right]   
 \\
  \times (\min(v(p),1-v(p))^{2(A+1)f_{m}}.\end{eqnarray*}
Using~(\ref{e:merci-alex-2}), we deduce that, for every $i \in \ig  m-c_{m}+1, m-c_{m} + 2(A+1)f_{m} \id$, 
\begin{eqnarray}\label{e:merci-alex-3} \ \ \hat{P} \left[g_{m}(k) \leq   S_{k} \leq d_{m}(k) ; \ k \in \ig 0, i\id  , \ S_{i}=x  \right] \geq
  \exp(- \zeta_{1} m^{1/3}) \\ \nonumber \times  (1/a_{m})  (\min(v(p),1-v(p))^{(2(A+1)+1)f_{m}}.\end{eqnarray}

Then an argument quite similar to that leading to~(\ref{e:merci-alex-2}) yields that, 
for any $i \in \ig  m-c_{m}+2(A+1)f_{m}+1, m \id$, 
\begin{eqnarray}\label{e:merci-alex-4} \ \  \hat{P} \left[g_{m}(k) \leq   S_{k} \leq d_{m}(k) ; \ k \in \ig 0, i\id  ,
 \ S_{i}=x  \right] \geq
 \exp(- \zeta_{1} m^{1/3})   \\ \nonumber  \times  (1/ ((A+1)a_{m}+1)) (\min(v(p),1-v(p))^{(A+1)f_{m}}.\end{eqnarray}
The conclusion follows by using~(\ref{e:change-of-measure-2}) 
and~(\ref{e:merci-alex-1}), (\ref{e:merci-alex-2}), (\ref{e:merci-alex-3}), (\ref{e:merci-alex-4}).
\end{proof}

\begin{lemma}\label{l:descente-importante}
There exists $\phi(A)>0$ such that, as $m$ goes to infinity,  
$$\hat{P} \left[\hat{g}_{m}(k) \leq   \hat{S}_{k} \leq \hat{d}_{m}(k) ; \ k \in \ig 0, m\id , \ \hat{S}_{m} \leq -(A/2)a_{m} \right] 
\geq \phi(A) \exp(-\zeta_{1} m^{1/3}).$$
\end{lemma}

We shall need the following elementary lemma.
\begin{lemma}\label{l:proba-de-descendre}
Consider a random walk $(Z_{i})_{i \geq 0}$ on $\R$, defined by $Z_{i}:=Z_{0}+\varepsilon_{1} + \cdots + \varepsilon_{i}$ for $i \geq 1$. Assume that the 
increments $\varepsilon_{i}$ are i.i.d. with respect to some probability measure $Q$ and satisfy $E(\varepsilon_{1})=0$ and $0<Var(\varepsilon_{1})<+\infty$.  
Then there exists $\phi(A) > 0$ such that, for all $m$ large enough,  on the event $a_{m}/4 \leq Z_{0} \leq 3a_{m}/4$, 
$$Q \left[ \left.  -Aa_{m}  \leq Z_{i} \leq a_{m} ;  \ i \in \ig  1, c_{m} \id , \  Z_{c_{m}} \leq -(A/2)a_{m}  \right| Z_{0}\right]  \geq \phi(A).$$
\end{lemma}

\begin{proof}[Proof of Lemma~\ref{l:proba-de-descendre}]
We re-use the notations of the proof of Lemma~\ref{l:proba-de-rester}.
The only point is to note that, by convergence to the Brownian motion, there exists $\phi(A)>0$ such that 
for all large $m$, 
$$Q \left[ -3Aa_{m}/4  \hspace{-1mm} \leq  \hspace{-1mm}  Z_{i}-Z_{0} \hspace{-1mm} \leq
\hspace{-1mm} a_{m}/4 ;  
 \ i \in \ig  1, c_{m} \id , \  Z_{c_{m}}-Z_{0} \hspace{-1mm} \leq \hspace{-1mm} - (A/2+3/4)a_{m} \right]  \geq  \phi(A).$$ 
The result follows easily (using the fact that $A$ is assumed to be $\geq 4$).
\end{proof}

\begin{proof}[Proof of Lemma~\ref{l:descente-importante}]
From the proof of Lemma~\ref{l:proba-totale}, we see that 
$$\hat{P} \left[g_{m}(k) \hspace{-1mm} \leq \hspace{-1mm}  S_{k}\hspace{-1mm} \leq \hspace{-1mm} d_{m}(k) ;  k \in \ig 0, m-c_{m}\id ,  a_{m}/4  \hspace{-1mm}\leq \hspace{-1mm}\hat{S}_{m-c_{m}}\hspace{-1mm} \leq \hspace{-1mm}3a_{m}/4 \right] 
\geq \exp(-\zeta_{1} m^{1/3}).$$
Then, by Lemma~\ref{l:proba-de-descendre}, as $m$ goes to infinity, 
$$\hat{P} \left[ \left. \hat{g}_{m}(k) \hspace{-1mm} \leq  \hspace{-1mm} \hat{S}_{k}\hspace{-1mm} \leq \hspace{-1mm}\hat{d}_{m}(k) ;  k \in \ig m-c_{m}+1, m\id ,  \hat{S}_{m}\hspace{-1mm} \leq \hspace{-1mm} -Aa_{m}/2 \right| a_{m}/4 \hspace{-1mm} \leq \hspace{-1mm}\hat{S}_{m-c_{m}}\hspace{-1mm} \leq \hspace{-1mm}3a_{m}/4\right]$$
is larger than or equal to $\phi(A)$. The conclusion follows.
\end{proof}

\begin{proof}[Proof of Proposition~\ref{p:choix-de-seq-admissible}]
Properties (i) and (ii) are immediate consequences of the definition.
Property (iii) is a direct consequence of Lemma~\ref{l:proba-eps}, letting $\chi(A):=\zeta_{2}(A)$.
From the definition, $g_{m}(k) \geq v(p)(k+1)$ for all $ k \in \ig q,  m-c_{m} \id$. Then, given $A$, for all large enough $m$, 
it is easily checked that $g_{m}(k) \geq (v(p) - 2Am^{-2/3})(k+1)$ for all $ k \in \ig m-c_{m}+1,  m \id $. This yields Property (iv). 
As for Property (v), consider~(\ref{e:change-of-measure-1}). Clearly, since $\gamma<1$,
$$\nu_{m}(\Z) \geq \gamma^{-(A/2)a_{m}}\hat{P} \left[\hat{g}_{m}(k) \leq   \hat{S}_{k} \leq \hat{d}_{m}(k) ; 
\ k \in \ig 0, m\id , \ \hat{S}_{m} \leq -(A/2)a_{m} \right].$$
 From Lemma~\ref{l:descente-importante}, the probability in the r.h.s. of the above expression is $\geq \phi(A) \exp(- \zeta_{1}m^{1/3})$, so that, choosing $A$ large enough,  
the term $\gamma^{-(A/2)a_{m}}$ dominates for large $m$. As a consequence, for such an $A$, $\nu_{m}(\Z)\geq 2008 + 1$ as soon as $m$ is large enough.
\end{proof}

\begin{proof}[Proof of Theorem~ \ref{t:BD-lower}]
Consider a parameter $\theta>0$, and let $m$ depend on $N$ in the following way: $m:=\floor{ \theta \log(N)}^{3}$. 
For $N$ large enough, we can apply Proposition~\ref{p:choix-de-seq-admissible}, and use Lemma~\ref{l:utilisation-admissible}, 
which states that$$ \P(B) \leq 
2 K \exp\left(- N \beta^{-1} \epsilon  p \delta^{2} \right).$$
It is easily checked from the definition that $K \leq ((A+1)a_{m}+1) m$.
 Note also that $\delta \sim \log(\beta)/m$ for large $N$. 
 Finally, $\epsilon = \exp(-\chi(A) m^{1/3})$. Choosing $\theta$ small enough, we check that
  $m\P(B) << m^{-2/3}$ as $N$ goes to infinity. Proposition~\ref{p:borne-inf-B} then implies that $v_{N}(p) \geq v(p) - 2Am^{-2/3}(1+o(1))$.
The result follows.
\end{proof}

\section{The case $1/2 \leq p < 1$}\label{s:autres-cas}

In the case $1/2 \leq p < 1$, it turns out that $v_{\infty}(p)=v(p)=1$, which makes it much easier to obtain estimates.

\subsection{Upper bound when $p=1/2$}

It is easily checked that, for all $m \geq 0$, the number 
of particles at position exactly $m$ after $m$ steps, that is, $X^{N}_{m}(m)$, is stochastically dominated by the total population at the $m-$th
generation of a family of $N$ independent Galton-Watson trees, with offspring distribution binomial$(2,1/2)$.  This corresponds to 
the critical case of Galton-Watson trees, and the probability that such a tree survives up to the $m-$th generation is $\leq c m^{-1}$ for some constant $c>0$
and all large $m$. As a consequence, for large enough $m$, 
$\P(X^{N}_{m}(m) \geq 1) \leq \E(X^{N}_{m}(m)) \leq cNm^{-1}$.

On the other hand, we have by definition that $m^{-1}\E \max(X^{N}_{m}) \leq   1 -   \frac{1}{m}  \P(X^{N}_{m}(m) = 0)$.
Choosing $m := A N$, where $A \geq 1$ is an integer, we see that, for large $N$,
$    m^{-1}\E \max(X^{N}_{m}) \leq 1 - 1/AN (1-c/A)$. The upper bound in~(\ref{e:un-demi}) follows by choosing $A>c$.
 
\subsection{Lower bound when $p=1/2$}

Given $m\geq 1$, define $U:=\inf \{ n \in \ig 1, m \id ; \ X^{N}_{n}(n) \leq 2N/3 \}$, with the convention that
 $\inf \emptyset := m$. Let $D$ denote the event that $\min X^{N}_{U} < U-1$. 
 
Using an argument similar to the proof of Proposition~\ref{p:borne-inf-B}, with $U$ and $D$ in place of 
$L$ and $B$ respectively, we deduce that
\begin{equation}\label{e:min-un-demi} v_{N}(p) \geq   1 - \frac{1}{\E(D)} - m \P(D).\end{equation}
 
 The lower bound  in~(\ref{e:un-demi}) is then a consequence of the two following claims.

 Our first claim is that there exists $c>0$ such that $\P(D) \leq \exp(- c N)$ for all large $N$. 
Recall the definition of $N^{\ell}_{k}(x)$ from the proof of Lemma~\ref{l:utilisation-admissible}, and 
choose $\delta$ small enough so that $(1-\delta) 4N/3>N$. It is easily seen that
$D \subset \{ N^{1}_{U-1}(U-1) \leq (1-\delta)2N/3      \}  \cup \{   N^{0}_{U-1}(U-1) \leq (1-\delta)2N/3   \}$.
Now, by definition, one has that $X^{N}_{U-1}(U-1) \geq 2N/3$, so that,  by the bound of Lemma~\ref{l:dev-bino},  conditional on $X^{N}_{U-1}$, 
the probabilities of $N^{1}_{U-1}(U-1) \leq (1-\delta)2N/3$ and  $N^{0}_{U-1}(U-1) \leq (1-\delta)2N/3$ are both  $\leq \exp(-c(\delta) N)$.
The bound on $\P(D)$ follows.   

Our second claim is that, for small enough $\epsilon>0$, with $m:=\floor{\epsilon N}$,  there exists $c(\epsilon)>0$ such that
$\E(D) \geq  c(\epsilon)N$ for all large $N$. To prove it, introduce the Markov chains
   $(V_{k})_{k \geq 0}$ and $(Z_{k})_{k \geq 0}$ on $\N$, defined as follows.
 First, $V_{0} \in \N$, and, given $V_{0},\ldots, V_{k}$, the next term $V_{k+1}$ is the minimum of $N$ and of a random variable with a binomial$(2V_{k},1/2)$
 distribution. On the other hand,  $Z_{0} \in \N$, and, given $Z_{0},\ldots, Z_{k}$,the distribution of  $Z_{k+1}$ is binomial$(2Z_{k},1/2)$.
Observe that the sequence $(X^{N}_{n}(n))_{n \geq 0}$ is a version of $V$ started at $V_{0}:=N$.  
Now, it is easily seen that, given two starting points $x,y \in \N$ such that $x \leq y$, one can couple two versions of $V$ starting from $x$ and $y$ 
respectively, in such a way that the version starting from $y$ is always above the version starting from $x$.  As a consequence, $U$ stochastically 
dominates  the random variable $T:=\inf \{ n \in \ig 1, m \id ; \ V_{n} \leq 2N/3 \}$ (again with $\inf \emptyset := m$), where $V$ is started at $V_{0}:=\floor{3N/4}$.
 Then observe that the distributions of $V$ and $Z$  started with $V_{0}:=Z_{0}:=\floor{3N/4}$, considered up to the hitting time of $\ig N, +\infty \ig$, coincide. 
As a consequence, the probabilities of the events $\{    \sup_{k \in \ig 0 , m \id}     |   V_{k}  -  \floor{3N/4}     |   \geq N/16    \}$
and $\{    \sup_{k \in \ig 0 , m \id}     |   Z_{k}  -  \floor{3N/4}     |   \geq N/16    \}$ coincide, and  the 
first of these two events implies that $T  = m$. 
 Now, $(Z_{k})_{k \geq 0}$ is a martingale, so that, by Doob's maximal inequality,
  $P \left(  \sup_{k \in \ig 0 , m \id}     |   Z_{k}  -  \floor{3N/4}     |   \geq N/16    \right) \leq E (Z_{m} - \floor{3N/4})^{2}   (N/16)^{-2}$.
  Then, it is easily checked from the definition that $E(Z_{k+1}^{2} | Z_{k}) = Z_{k}^{2} + Z_{k}/2$, and, 
  using again the fact that $(Z_{k})_{k \geq 0}$ is a martingale,
   we deduce that  $E (Z_{m} - \floor{3N/4})^{2} \leq m N/2$. As a consequence, we see that, choosing $\epsilon>0$ small enough, 
 we can ensure that   $P \left(  \sum_{k \in \ig 0 , \floor{ \epsilon N} \id}     |   Z_{k}  -  \floor{3N/4}     |   \geq N/16    \right) \leq 1/2008$
 for all large $N$. For such an $\epsilon$, and all $N$ large enough, we thus have that 
 $\P(U = m) \geq P(T=m) \geq 1/2008$. The conclusion follows.
 
 \subsection{Upper and lower bound when $1/2<p<1$}

As for the upper bound, observe that asking all the $2N$ particles generated during the branching step to remain at the position from which they are
originated has a probability equal to at most $(1-p)^{2N}$, so that $\E( \max X_{n} ) \leq n (1 - (1-p)^{2N})$. 
As for the lower bound, observe that, starting from $N$ particles at a site, the number of particles generated from these during a branching step and 
that perform  $+1$ random walk steps has a binomial$(2N,p)$ distribution, whose expectation is $2p N$, with $2p > 1$. Using the bound stated in 
Lemma~\ref{l:dev-bino}, we see that the probability for this number to be less than $N$ is $\leq \exp(-c N)$ for some $c>0$. As a consequence, 
 $\E( \min X_{n} ) \geq n (1 - \exp(-c N))$. 

\section{Concluding remarks}\label{s:remarks}

\begin{remark}
Can we derive a reasonably simple explanation of how the $\log(N)^{-2}$ arises,  based on the mathematical proofs presented above ?
Broadly speaking, the key point in both the upper and the lower bound seems to be the following (we re-use some notations from Section~\ref{s:lower}): consider a large integer 
$m$ and look for a scale $\Delta$ such that  
\begin{equation}\label{e:bouc}2^{m} P( |\hat{S}_{1}|,\ldots, |\hat{S}_{m}| \propto \Delta     )  \asymp 1/N.\end{equation} 
In view of the change of measure, and of the fact
that, by Brownian scaling, $ \hat{P}(|\hat{S}_{1}|,\ldots, |\hat{S}_{m}| \propto \Delta) \asymp \exp(-m/\Delta^{2})$, we see that there are two 
factors involved in the above probability: $\gamma^{\pm \Delta}$, and $\exp(-m/\Delta^{2})$.  
Equating the exponential scales of these two factors yields $\Delta \propto m^{1/3}$, and~(\ref{e:bouc}) then implies that $m \propto \log(N)^{3}$, whence an 
 average velocity shift over the $m$ steps of order $\Delta/m \propto \log(N)^{-2}$.
\end{remark}

\begin{remark}
What we have proved is that the order of magnitude of $v_{N}(p) - v_{\infty}(p)$ is indeed $\log(N)^{-2}$. It would of course be quite interesting to get more precise
asymptotic results for this quantity, since there is at least compelling numerical evidence that $v_{N}(p) - v_{\infty}(p) \sim c(p) \log(N)^{-2}$ for some constant $c$.
\end{remark}

\begin{remark}
Both the upper and lower bound presented here relie upon controlling the behavior of the particle system for time intervals of length $m \propto \log(N)^{3}$. 
This is the same order of magnitude as the one observed for the coalescence times of the genealogical process underlying the branching-selection algorithm, 
from empirical studies and heuristic arguments (see e.g.~\cite{BruDerMueMun}).
Although we do not know how to establish a rigorous relationship between these facts,  this  at least provides another indication 
that the $\log(N)^{3}$ time scale is particularly relevant for the study of these systems.
\end{remark}

\bibliographystyle{plain}
\bibliography{BD-particle}

\begin{thebibliography}{10}

\bibitem{BenDep2}
R.~Benguria and M.~C. Depassier.
\newblock On the speed of pulled fronts with a cutoff.
\newblock {\em Phys. Rev. E}, 75(5), 2007.

\bibitem{BenDep1}
R.~Benguria and M.~C. Depassier.
\newblock Validity of the {B}runet-{D}errida formula for the speed of pulled
  fronts with a cutoff.
\newblock {\em arXiv:0706.3671}, 2007.

\bibitem{BruDerMueMun}
{\'E}.~Brunet, B.~Derrida, A.~H. Mueller, and S.~Munier.
\newblock Effect of selection on ancestry: an exactly soluble case and its
  phenomenological generalization.
\newblock {\em Phys. Rev. E (3)}, 76(4):041104, 20, 2007.

\bibitem{BruDer1}
Eric Brunet and Bernard Derrida.
\newblock Shift in the velocity of a front due to a cutoff.
\newblock {\em Phys. Rev. E (3)}, 56(3, part A):2597--2604, 1997.

\bibitem{BruDer2}
{\'E}ric Brunet and Bernard Derrida.
\newblock Microscopic models of traveling wave equations.
\newblock {\em Computer Physics Communications}, 121-122:376--381, 1999.

\bibitem{BruDer3}
{\'E}ric Brunet and Bernard Derrida.
\newblock Effect of microscopic noise on front propagation.
\newblock {\em J. Statist. Phys.}, 103(1-2):269--282, 2001.

\bibitem{ConDoe}
Joseph~G. Conlon and Charles~R. Doering.
\newblock On travelling waves for the stochastic
  {F}isher-{K}olmogorov-{P}etrovsky-{P}iscunov equation.
\newblock {\em J. Stat. Phys.}, 120(3-4):421--477, 2005.

\bibitem{DumPopKap}
Freddy Dumortier, Nikola Popovi{\'c}, and Tasso~J. Kaper.
\newblock The critical wave speed for the
  {F}isher-{K}olmogorov-{P}etrowskii-{P}iscounov equation with cut-off.
\newblock {\em Nonlinearity}, 20(4):855--877, 2007.

\bibitem{McD}
Colin McDiarmid.
\newblock Concentration.
\newblock In {\em Probabilistic methods for algorithmic discrete mathematics},
  volume~16 of {\em Algorithms Combin.}, pages 195--248. Springer, Berlin,
  1998.

\bibitem{MueMytQua}
C.~Mueller, L.~Mytnik, and J.~Quastel.
\newblock Small noise asymptotics of traveling waves.
\newblock {\em Markov Process. Related Fields}, 14, 2008.

\bibitem{Pem}
R.~Pemantle.
\newblock Search cost for a nearly optimal path in a binary tree.
\newblock {\em arXiv:math/0701741}, 2007.

\end{thebibliography}

\end{document}